\documentclass[11pt,a4paper
]{article}

\usepackage[top=12mm,bottom=18mm,left=30mm,right=30mm,head=12mm,includeheadfoot]{geometry}

\usepackage[utf8]{inputenc}
\usepackage[T1]{fontenc}
\usepackage[english]{babel}
\usepackage{amsmath,amssymb,amsthm}
\usepackage{mathrsfs,dsfont} 
\usepackage[bibstyle=numeric, backend=biber, sorting=nyt, citestyle=numeric-comp, giveninits, url=false, isbn=false, date=year]{biblatex} 
\usepackage{mathtools, ifthen, xparse}
\usepackage[shortlabels]{enumitem}
\usepackage{csquotes}\MakeOuterQuote{"}
\usepackage{tikz,tikz-cd}
\usepackage[colorlinks]{hyperref}
\usepackage[capitalize]{cleveref}
\usepackage{soul}

\makeatletter
\AtBeginDocument{\hypersetup{
        pdftitle = {\@title},
        pdfauthor = {\@author}
        verbose=false,
        colorlinks,
        allcolors={orange},
        urlcolor={blue}
}}\makeatother

\newtheorem{theorem}{Theorem}
\newtheorem*{theorem*}{Theorem}
\newtheorem{lemma}[theorem]{Lemma}
\newtheorem{corollary}[theorem]{Corollary}

\newtheorem{proposition}[theorem]{Proposition}
\newtheorem*{problem*}{Problem}
\newtheorem{definition}[theorem]{Definition}

\theoremstyle{definition}
\newtheorem{remark}[theorem]{Remark}
\newtheorem{example}[theorem]{Example}
\theoremstyle{remark}
\newtheorem*{example*}{Example}
\newtheorem*{remark*}{Remark}

\crefname{lemma}{Lemma}{Lemmas}

\renewcommand{\bar}[1]{\overline{#1}}
\newcommand\CC{\mathbb C}

\newcommand\eps\varepsilon

\newcommand\NN{\mathbb N}
\def\placeholder{\,\cdot\,}

\renewcommand{\S}{{\mathcal S}}
\newcommand\restrictedto\upharpoonright

\newcommand\oo\infty
\newcommand\ox\otimes

\NewDocumentCommand\TC{o}{{\IfNoValueTF{#1}{\mathcal{T}(\mathcal{H})}{\mathcal{T}(#1)}}}
\newcommand\boundedoperatorsymbol{{\mathcal B}}
\NewDocumentCommand\BO{mo}{ \IfNoValueTF{#2}{\boundedoperatorsymbol(#1)}{\boundedoperatorsymbol(#1,#2)}}

\def\up#1{^{(#1)}}

\let\mc\mathcal

\renewcommand\limsup\varlimsup
\renewcommand\liminf\varliminf
\DeclareMathOperator\id{id}

\DeclareMathOperator{\Sp}{Sp}

\newcommand*\quotient[2]{{^{\textstyle #1}\big/_{\textstyle #2}}}

\let\lim\relax
\NewDocumentCommand\lim{o}{\IfNoValueTF{#1}{\mathop{\textup{lim}}}{{#1}\!-\!\mathop{\textup{lim}}}}

\DeclareFontFamily{U}{matha}{\hyphenchar\font45}
\DeclareFontShape{U}{matha}{m}{n}{ <-6> matha5 <6-7> matha6 <7-8> matha7 <8-9> matha8 <9-10> matha9 <10-12> matha10 <12-> matha12 }{}
\DeclareSymbolFont{matha}{U}{matha}{m}{n}
\DeclareFontFamily{U}{mathx}{\hyphenchar\font45}
\DeclareFontShape{U}{mathx}{m}{n}{ <-6> mathx5 <6-7> mathx6 <7-8> mathx7 <8-9> mathx8 <9-10> mathx9 <10-12> mathx10 <12-> mathx12 }{}
\DeclareSymbolFont{mathx}{U}{mathx}{m}{n}
\DeclareMathDelimiter{\vvvert} {0}{matha}{"7E}{mathx}{"17}%
\DeclarePairedDelimiterX{\normiii}[1]{\vvvert}{\vvvert} {\ifblank{#1}{\:\cdot\:}{#1}}

\DeclarePairedDelimiter\paren\lparen\rparen

\DeclarePairedDelimiterX\norm[1]\lVert\rVert{\ifblank{#1}{\placeholder}{#1}}
\DeclarePairedDelimiter\abs\lvert\rvert

\DeclarePairedDelimiterX\ip[2]{\langle}{\rangle}{#1 , #2}

\providecommand\given{}  
\newcommand{\SetSymbol}[1][]{ \nonscript\ #1\vert \allowbreak \nonscript\ \mathopen{} } 
\DeclarePairedDelimiterX{\set}[1]\{\}{ \renewcommand\given{\SetSymbol[\delimsize]} #1 }

\newcommand{\sbt}{\,\begin{picture}(-1,1)(-1,-2.5)\circle*{3}\end{picture}\,\,\, }

\renewcommand\limsup\varlimsup
\renewcommand\liminf\varliminf

\newcommand\A{\mathcal A}

\renewcommand\S{\mathcal S}     
\newcommand\T{\mathcal T}     

\newcommand\C{\mathbf{C}}      
\newcommand\nets{{\mathbf{N}}} 
\newcommand{\Q}{{\mathbf Q}}   
\def\oo{\infty}
\def\J{{\mathcal J}}

\let\d\relax
\def\blob{{\hspace{0.0pt}{\mathbin{{\hbox{\scalebox{0.7}{$\sbt$}}}}}\hspace{+0pt}}}
\def\d{_{\blob{\hspace{-1pt}}}} 

\def\j#1#2{j_{#1#2}}
\def\jt{\ensuremath{\j{}{}}}      
\def\jlim{\jt\mkern-1mu\hbox{-}\mkern-3mu\lim}

\def\ojt{{\ensuremath{\tilde\jmath}}}
\def\oj#1#2{\ojt_{#1#2}}
\def\ojlim{\ojt\mkern-0mu\hbox{-}\mkern-3mu\lim}

\def\klim{k\mkern-1mu\hbox{-}\mkern-3mu\lim}

\def\jnu#1#2{j^{(\nu)}_{#1#2}}
\def\jtnu{\ensuremath{\jnu{}{}}}   
\def\jnulim{\jtnu\mkern-1mu\hbox{-}\mkern-3mu\lim}

\def\jtmu{\ensuremath{\jnu{}{}}}   

\def\Cstar{\mathrm{C}^*}

\renewcommand\TC{{\mathfrak T}}
\newcommand\TT{{\mathcal T}}

\let\seminorm\normiii

\title{Soft inductive limits of operator systems and a noncommutative Lazar-Lindenstrauss theorem}
\author{Kristin Courtney\footnotemark[1], Niklas Galke\footnotemark[2], Lauritz van Luijk\footnotemark[3], Alexander Stottmeister\footnotemark[3]} 

\date{\footnotemark[1] University of Southern Denmark, Department of Mathematics and Computer Science,\\
Campusvej 55, 5230 Odense M,
Denmark\\[3pt]
\footnotemark[2] Universitat Autònoma de Barcelona, Departament de Física, GIQ, \\08193 Bellaterra, Barcelona, Spain \\[3pt]
\footnotemark[3] Leibniz University Hanover, Institut f\"ur Theoretische Physik,\\ Appelstrasse 2, 30167 Hannover, Germany
\\[11pt]
\today
}

\begin{document}

\maketitle

\begin{abstract}
    We establish a flexible generalization of inductive systems of operator systems, which relaxes the usual transitivity (or coherence) condition to an asymptotic version thereof and allows for systems indexed over arbitrary nets. 
    To illustrate the utility of this generalization, we highlight how such systems arise naturally from completely positive approximations of nuclear operator systems. 
    Going further, we utilize an argument of Ding and Peterson to show that a separable operator system is nuclear if and only if it is an inductive limit of matrix algebras, generalizing a classic Theorem of Lazar and Lindenstrauss to the setting of noncommutative Choquet theory.
\end{abstract}

\section{Introduction}
Inductive limit constructions have played a pivotal role in operator theory because they allow for the study of complicated objects via more elementary building blocks. 
This construction has been heavily utilized in $\Cstar$-algebras since Glimm \cite{Glimm} and has more recently gained significant interest in the study of operator systems, starting with Kirchberg's treatment in \cite{Kir95} and later and much more systematically in \cite{MT18}.

In physics, in particular, in quantum theory, inductive limits have a similarly long tradition, where they have been used to describe various limiting situations, such as the classical limit \cite{Hepp1974,werner_classical_1995}, the thermodynamic limit \cite{bratteli1987oa1,bratteli1997oa2,naaijkens_quantum_2017}, mean-field theory \cite{hepp_equilibrium_1973,hepp_phase_1973,RAGGIO1989,DUFFIELD1992,DUFFIELD1992b,drago_strict_2024}, scaling limits \cite{Evenbly2009,pfeifer_entanglement_2009,haegeman_calculus_2013,evenbly_real-space_2014,Milsted2018,osborne_quantum_2019,stottmeister_operator-algebraic_2021,morinelli_scaling_2021,van_nuland_strict_2022,osborne_conformal_2023} as well as large-scale entanglement structures \cite{van_luijk_multipartite_2025,luijk_large-scale_2025,chemissany_infinite_2025} (see \cite{jdynamics} for a general perspective). 

In this regard, it is also worth highlighting the interplay between mathematics and physics that has stimulated important results, notably Connes' classification of injective von Neumann algebras \cite{Powers1967,Araki1968,connes1976injective,haagerup_uniqueness_1987}, and fruitful research directions, for example, the recent investigations into conformal field theory and Thompson's groups \cite{Jones2017,Jones2017a,Jones2017b,Osborne2019,Brothier2019b,brothier_operator-algebraic_2020, aiello_spectral_2021,brothier_haagerup_2023}.

Also in Choquet theory, a classical result of Lazar and Lindenstrauss \cite{LL71}, which says that every Choquet simplex is the projective limit of finite-dimensional simplices, is obtained as a corollary to the dual inductive limit result for $\ell^\infty_n$ spaces. This has had a wide breadth of applications, particularly in the case of $\Cstar$-algebras whose trace spaces are Choquet simplices \cite{Thoma64,sakai_c-algebras_1998}. 

In order to broaden the scope of inductive limits, one can demand more flexibility in the construction, 
leading to generalized inductive limit constructions starting with Blackadar and Kirchberg in \cite{blackadar1997generalized} and seen much more recently in \cite{jdynamics,CW1,C23}: The $^*$-homomorphic connecting maps of classical inductive systems are replaced with maps that only asymptotically preserve the $\Cstar$-algebraic structure. 
The inductive systems considered in \cite{blackadar1997generalized,jdynamics,CW1,C23} have completely positive connecting maps and hence (in the unital setting) can actually be seen as inductive systems of operator systems which are (completely order isomorphic to) $\Cstar$-algebras (\cite[Theorem 3.7]{C23}, \cite[Proposition 2.7]{CW1}). 
These are of independent interest in the operator system setting because they can characterize when certain operator systems are completely order isomorphic to $\Cstar$-algebras. Nonetheless, because these completely positive maps are asymptotically encoding a $\Cstar$-algebraic structure, the generalizations in \cite{blackadar1997generalized,jdynamics,CW1,C23} are not applicable to operator systems at large.  


A further relaxation of the inductive limit construction was considered in \cite{jdynamics}, building on earlier work in \cite{DUFFIELD1992,werner_classical_1995}, as it is too restrictive to allow for an effective description of certain limits in physical theories. This form of generalized inductive limits was coined \emph{soft inductive} systems by the third and fourth authors, along with R.~F.~Werner.
Soft inductive systems replace generalized inductive sequences with nets and relax the usual transitivity condition $\j nm \j ml =\j nl$, $n<m<l$, (called \emph{coherenence} in \cite{blackadar1997generalized}) to asymptotic transitivity. 
Like inductive systems, they can be formulated in various categories and are compatible with additional structures. The softness criterion is adaptable to the operator system setting, and in fact, as Theorem \ref{thm:nuclearity_iff} shows, is already a naturally occurring phenomenon therein. The first aim of this article is to initiate the study of soft inductive limits of operator systems.

Our second objective is to apply our results to noncommutative Choquet theory, which stems from the work of Wittstock, Effros, Webster, and Winkler \cite{Wit81,Wit84,EW97,WW99} and has since become an increasingly fruitful perspective in operator systems \cite{Davidson2016, DK21, DK24,KKM,KS22}.  
Here, We contribute to the theory by providing a noncommutative analogue to the aforementioned classical result of Lazar and Lindenstrauss 
along with its converse due to Jellet \cite{J68} and Davies--Vincent-Smith \cite{DVS68}.

To set up this generalization, we must first translate the classical result into the operator system setting. For a compact, convex, metrizable set $K$, we consider the $^*$-vector space of its continuous $\mathbb{C}$-valued affine functions $A(K)$. 
This is a function system, meaning an Archimedean order unit $^*$-vector space, which we also view as an operator subsystem of $C(K)$.\footnote{Metrizability translates to the separability of $C(K)$ and $A(K)$. }
As proved in \cite{VT09}, over $\mathbb{C}$ we still have Kadison's representation theorem as enjoyed in the real setting: $S(A(K))$ is affinely homeomorphic to $K$ via point evaluations. Hence dualizing the theorem of Lazar and Lindenstrauss (and its converse) over $\mathbb{C}$ tells us that a compact, convex, metrizable space $K$ is a simplex if and only if $A(K)$ is the inductive limit of a sequence of affine function spaces of finite-dimensional simplices, $A(\Delta^n)$, which we identify with $\mathbb{C}^{n+1}$. 
To take this one step further into operator system language, we recall a theorem of Namioka and Phelps \cite{NP69} which says that a compact, convex metrizable $K$ is a simplex precisely when $A(K)$ is nuclear in the function system category. Thus, the theorem of Lazar and Lindenstrauss and its converse can be restated as: a separable function system $A(K)$ is nuclear precisely when it is the inductive limit of some sequence $(\mathbb{C}^{n_k})_k$. 

Moving out of the classical picture, there is a well-established notion of noncommutative (nc) convex sets, which has been developed over the years in \cite{Wit81,Wit84,EW97,WW99} and more recently in \cite{DK24,Davidson2016,DK21,KS22,KKM}.
For a detailed discussion, we refer to see \cref{sect. nc}.
In the following, we briefly sketch the concept, guided by Kadison's representation theorem in \cite{DK24}:

The categorical duality is between closed operator systems with unital completely positive maps and compact nc convex sets with continuous nc affine functions. 
In \cite{KS22}, Kennedy and Shamovich expanded the theory by introducing the noncommutative analogue of Choquet simplices. They proved that these correspond to what Kirchberg dubbed $\Cstar$-systems, i.e., operator systems whose double dual is a von Neumann algebra, in particular, including all nuclear operator systems (by deep results of  \cite{CE77,EOR01,Kir95}). 

\subsection{Summary of results}

\cref{sec:soft_ind} establishes the necessary framework for soft inductive systems for normed spaces, which is then adapted to operator systems in \cref{sec:op_sys}. 
In \cref{sec:nuclear}, we show that every nuclear operator system is a soft inductive limit of a net of finite-dimensional von Neumann algebras and that this can be strengthened to a (strict) inductive limit in the separable setting.
As a converse to the latter, we prove in Theorem~\ref{thm: strict iff} that a separable operator system is nuclear if and only if it is the (strict) inductive limit of finite-dimensional von Neumann algebras.

In \cref{sect. nc}, we interpret Theorem~\ref{thm: strict iff} in the setting of noncommutative Choquet Theory by exploiting the aforementioned duality between closed operator systems and compact nc convex sets:
The dual of an operator system $S$ is its noncommutative (nc) state space (consisting of all ucp maps $S\to M_n(\mathbb{C})$ for all $n$ up to a sufficiently large cardinal). 
According to \cite{KS22}, a compact nc convex set is a Choquet simplex, precisely when it is the nc state space of a $\Cstar$-system $\mathcal{S}$. 
Combining \cite{CE77,EOR01,Kir95}, these notably include the nc state spaces of nuclear operator systems, which we call \emph{nuclear nc simplices}. Among these are the nc state spaces of finite-dimensional von Neumann algebras, which serve as the noncommutative analogue of a finite-dimensional simplex. 
Thus, by dualizing Theorem~\ref{thm: strict iff}, we to obtain our noncommutative analogue to the Lazar-Lindenstrauss theorem (Corollary~\ref{cor: strict iff*}): (under suitable separability assumptions) a compact nc convex set is a nuclear nc simplex if and only if it is the projective limit of nc state spaces of finite-dimensional von Neumann algebras.

\null

\paragraph{Acknowledgements.}
The authors thank Vern Paulsen, Matthew Kennedy, and Jesse Peterson for helpful discussions.
The authors thank Reinhard F.\ Werner for suggesting Lemma~\ref{lem:B-space limit}.

\paragraph{Funding.} 
AS and LvL have been funded by a Stay Inspired Grant of the MWK Lower Saxony (Grant ID: 15-76251-2-Stay-9/22-16583/2022). KC was supported by the Independent Research Fund Denmark through DFF grant 1026-00371B and through the Inge Lehmann Grant 10.46540/4303-00003B.

\section{Soft inductive systems}\label{sec:soft_ind}

Soft inductive systems were introduced in \cite{jdynamics} as a flexible generalization of inductive systems.
The core difference to inductive limits is that the transitivity condition $\j nm \j ml =\j nl$, $n<m<l$, is relaxed to asymptotic transitivity.
Like inductive systems, they can be formulated in various categories and are compatible with additional structures.
We briefly summarize the framework here and refer to \cite{jdynamics} for more details. 
Since we want to generalize to the category of operator systems (where it is desirable not to assume norm-completeness), we start by discussing soft inductive systems of normed spaces.
We will then consider the categories of operator systems and $\Cstar$-algebras.

\subsection{Normed spaces}\label{sec:norm_sp}

\begin{definition}
    A {\bf soft inductive system} $(E,j)$ of normed spaces over a directed set $(N,\le)$ is a net of normed spaces $\{E_n\}_{n\in N}$ together with linear contractions $\j nm:E_m\to E_n$ defined if $n>m$, which are asymptotically transitive:
    \begin{equation}
        \lim_{n\gg m} \norm{(\j nl-\j nm\j ml)a_l} = 0, \qquad l\in N,\ a_l\in E_l,
    \end{equation}
    where we use the notation $\lim_{n\gg m}$ as a shorthand for $\lim_m \limsup_n$.
    If each $E_n$ is a Banach space, we call $(E,j)$ a soft inductive system of Banach spaces. 
    A {\bf soft inductive sequence} of normed (Banach) spaces is a soft inductive system $(E,j)$ with index set $N=\NN$ (with the usual order). 
\end{definition}

We denote by $\nets(E\d)$ the space of uniformly bounded nets $a\d = (a_n)_{n\in N}$, $a_n\in E_n$ the norm $\norm{a\d}_\nets=\sup_n \norm{a_n}$ (this does not depend on the connecting maps $j_{nm}$).
In other words, $\nets(E\d)$ is nothing but the direct product: $\nets(E\d)=\prod_{n\in N}E_n$.
If necessary, we extend the connecting maps $\j nm$ to all indices by setting $\j nn \coloneqq \id_{E_n}$ and $\j nm=0$ if $n\ngeq m$.
Nets of the form $(\j nl a_l)_{n\in N}=:\j\blob l a_l$, $a_l\in E_l$, are called \emph{basic nets} and form a particularly important subset of $\nets(E\d)$.\footnote{The extension of the connecting maps to indices $n\ngeq m$ ensures that basic nets are elements of $\nets(\A\d)$.}
A net $x\d\in\nets(\A\d)$ is \emph{\jt-convergent} if 
\begin{equation}\label{eq:j_convergent}
    \lim_{n\gg m}\, \norm{a_n -\j nm a_m} =0
\end{equation}
and we write $\C(E\d,j)$ for the space of \jt-convergent nets.
\cref{eq:j_convergent} can be read as stating that basic nets are \jt-convergent.
Another class of nets $a\d\in\nets(E\d)$ for which \jt-convergence is immediate are \emph{null nets}, defined as nets $a\d$ such that $\lim_n\norm{a_n}=0$.
Null nets form a closed subspace, which we denote by $\C_0(\A\d)$.
$\nets(\A\d)$ is equipped with a natural seminorm
\begin{equation}\label{eq:seminorm}
    \seminorm{a\d} \coloneqq \limsup_n \,\norm{a_n} = \inf_{b\d\in\C_0(E\d)}\ \norm{a\d+b\d}_\nets.
\end{equation}
The second equality is proved in \cite[]{jdynamics}. 
Using this seminorm, we can rewrite \cref{eq:j_convergent} as $\lim_m \seminorm{a\d-\j\blob m a_m}=0$.
It can be checked from \cref{eq:j_convergent} that the limit $\lim_n \norm{a_n}$ exists for all $a\d\in\C(E\d,j)$ (and, hence, agrees with the seminorm).
Clearly, the null-space of this seminorm is precisely the space of null nets $\C_0(E\d)$.
The \emph{limit space} of a soft inductive system of normed spaces is constructed as the quotient
\begin{equation}\label{eq:j_}
    E_\oo =\quotient{\C(E\d,j)}{\C_0(E\d)}\ .
\end{equation}
As the quotient of a normed space by a closed subspace $E_\oo$ is itself a normed space and, by \eqref{eq:seminorm}, the norm is equal to $\norm{a_\oo}= \seminorm{a\d}$, $a\d\in\C(E\d,j)$.
The canonical contraction of $\C(E\d,j)$ onto the quotient $E_\oo$ will be denoted $\jlim$ and will often denote $\jlim a\d$ as $\jlim_n a_n$ or simply $a_\oo$.
Furthermore, we obtain induced contractions $\j\oo n:E_n\to E_\oo$ for all $n$, defined as the \jt-limit of basic nets starting at $n$, i.e.,
\begin{equation}\label{eq:j maps}
    \j\oo n: E_n \to E_\oo,\ a_n \mapsto \jlim_n \j\oo n a_n.
\end{equation}
It then follows that $(\j\oo n a_n)_n$ is a Cauchy net for all $a\d\in\C(E\d,j)$ which converges to $\jlim_n a_n$.
If each $E_n$ is a Banach space, then $E_\oo$ is also a Banach space (since it is the quotient of a Banach space by a closed subspace).
Let us mention the following, rather trivial, example:

\begin{example}[Completion]
    Let $E$ be a normed space, $N=\NN$ (or any other directed set not containing a largest element), and $\j nm = \id_E$.
    Then $\nets(E\d) = \ell^\oo(N,E)$ is the set of bounded sequences, $\C(E\d,j)=c(N,E)$ is the set of all Cauchy sequences in $E$ and $\C_0(E\d)=c_0(N,E)$ is the set of null sequences.
    Therefore, the above construction of the limit space $E_\oo=\C(E\d,j)/\C_0(E\d)$ amounts to Cauchy sequences modulo null sequences and we see that the limit space is the completion
    \begin{equation}
        E_\oo = \bar E.
    \end{equation}
\end{example}

In fact, the limit space will always be a Banach space. 
This follows from \cite{jdynamics} and the following Lemma, due to Reinhard F.\ Werner:

\begin{lemma}\label{lem:B-space limit}
    Let $(E,j)$ be a soft inductive system of normed spaces. 
    Denote by $\bar E_n$ the completion of $E_n$ and by $\bar j_{nm}:\bar E_m\to \bar E_m$ the continuous extension of $\j nm$.
    Then $(\bar E,\bar j)$ is a soft inductive system of Banach spaces, and there exists an isometric isomorphism $E_\oo \cong \bar E_\oo$ such that $\jlim_n a_n = \bar j\!\text-\!\lim_n a_n$, $a\d\in\C(E\d,j)$.
\end{lemma}

\begin{proof}
    It is clear that $(\bar E,\bar j)$ is a soft inductive system of Banach spaces.
    Hence, its limit space $\bar E_\oo$ is a Banach space ($\bar E_\oo$ denotes the limit space of $\bar E,\bar j)$, not the completion of $E_\oo$). 
    The embeddings $E_n\hookrightarrow \bar E_n$ induce isometric embeddings of $\nets(E\d)$, $\C(E\d,j)$ and $\C_0(E\d)$ into $\nets(\bar E\d)$, $\C(\bar E,\bar j)$ and $\C_0(\bar E\d)$. Since $\C_{0}(\bar{E}\d)=\overline{\C_{0}(E\d)}$, we get an embedding $E_\oo\hookrightarrow\bar E_\oo$.
    Given $a\d \in \C(\bar E,\bar j)$, we find for each $n\in N$ an element $b_n\in E_n$ such that $\seminorm{a\d-b\d}=\lim_n \norm{a_n-b_n}\to 0$ and $\sup_n \norm{b_n}<\oo$ because $E_{n}\subset\bar{E}_{n}$ is dense, i.e., $a\d-b\d\in\C_{0}(\bar{E},\bar{j})$, which implies that $b\d\in\C(E\d,j)$:
    \begin{align}
        \seminorm{b\d-\j\blob m b_{m}} & \leq \seminorm{b\d-a\d}+\seminorm{a\d-\bar{j}_{\blob m} a_{m}}+\seminorm{\bar{j}_{\blob m}(a_{m}-b_{m})} \nonumber \\
        & \leq \seminorm{a\d-\bar{j}_{\blob m} a_{m}} + \norm{a_{m}-b_{m}}_{m}.
    \end{align}
    By construction $a\d$ and $b\d$ differ by an element of $\C_{0}(\bar{E}\d)$, and we have $a_\oo = \bar{b}_\oo \in \bar{E}_\oo$, where we denote $\bar{b}_{\oo}=\bar{j}\!\!\;\text{-}\!\lim_{n}b_{n}$. Since $a\d$ was arbitrary, this shows $E_\oo = \bar E_\oo$.
\end{proof}

Given a soft inductive system of normed spaces, we denote 
\begin{equation}
    \Q(E\d) := \nets(E\d)/\C_0(E\d),
\end{equation}
and write $[a\d]$ for the image of $a\d\in\nets(E\d)$ in the quotient.
This is a Banach space in the quotient norm
\begin{equation}
    \norm{[a\d]} =\limsup_n \,\norm{a_n} = \inf_{b\d\in\C_0(E\d)}\ \norm{a\d+b\d}_\nets, \qquad a\d\in\nets(E\d).
\end{equation}
We define contractive linear maps  $\j\oo n\colon E_n\to \Q(E\d)$ by\begin{equation}\label{eq:j maps'}
    \j\oo n(x)=[j_{\blob n}(x)],\ x\in E_n
\end{equation} 
and set 
\begin{equation}\label{eq:same limits}
    E'_\infty:=\overline{\bigcup_n \j\oo n(E_n)}\subset \Q(E\d).
\end{equation}

Since $\C_0(E\d)$ and $\C(E\d,j)$ are closed in $\nets(E\d)$, we can identify $E_\infty=\C(E\d,j)/\C_0(E\d)$ isometrically as a closed subspace of $\Q(E\d)$ and view the map $\jlim: \C(E\d,j)\to E_\infty$ as the restriction of the quotient map $\nets(E\d) \to \Q(E\d)$. Then the maps $\j\oo n$ 
from \eqref{eq:j maps} and \eqref{eq:j maps'} agree.
It follows that $E'_\infty=\overline{\bigcup_n \j\oo n(E_n)}\subset E_\infty$. On the other hand, if  $x_\blob\in \C(E\d,j)$, then  
$[x_\blob]= \lim_n \j\oo n(x_n)\in \overline{\bigcup_n \j\oo n(E_n)}$.  Hence, $E'_\infty$ and $E_\infty$ agree as normed linear spaces and thus by Lemma~\ref{lem:B-space limit} as Banach spaces. Moreover, $\C(E\d,j)$ is the pre-image of $E'_\infty$ under the quotient map $\nets(E\d)\to \Q(E\d)$, and (from \cite[Eq.~(8)]{jdynamics}) $[x_\blob]=\jlim_n x_n$ for $x_\blob\in \C(E\d,j)$.

The limit space of a soft inductive system can be characterized uniquely by a universal property \cite[Prop.~13,14]{jdynamics}:

\begin{proposition}[Universal property]\label{prop:univ_prop}
    Let $(E,j)$ be a soft inductive system of normed spaces.
    Then $(E_\oo,\j\oo\blob)$ enjoys the following property:
    If $\alpha_n:E_n \to F$ is a net of linear contractions into some Banach space $F$ that maps \jt-convergent nets to Cauchy nets, then there exists a unique linear contraction $\alpha_\oo:E_\oo\to F$ such that $\alpha_\oo(\jlim_na_n)=\lim_n \alpha_n(a_n)$, $a\d\in\C(E\d,j)$.

    Let $\tilde E_\oo$ be a Banach space and $\oj\oo n:E_n\to\tilde E_\oo$ be linear contractions taking \jt-convergent nets to Cauchy nets. If the above property holds for $(\tilde E_\oo,\oj\oo \blob)$, then there exists an isometric isomorphism $\psi:E_\oo\to\tilde E_\oo$ with $\psi(\jlim_n a_n) = \lim_n \oj\oo n a_n$ extends to an 
\end{proposition}

The universal property follows from the following immediate observation:
If $(E,j)$ and $(F,k)$ are soft inductive systems of normed spaces over the same directed set and if $\alpha_n:E_n\to F_n$ is a net of linear contractions that takes \jt-convergent to $k$-convergent nets, there exists a limit contraction $\alpha_\oo:E_\oo\to F_\oo$ such that $\alpha_\oo(\jlim_na_n)=k\text-\!\lim_n\alpha_n(a_n)$.
The proof of uniqueness can be taken from \cite[Prop.~14]{jdynamics}.

For emphasis, we will refer to normal inductive systems, i.e., soft inductive systems such that $\j nm \j ml = \j nl$ for all $n>m>l$, as \emph{strict} inductive systems.
A soft inductive system is called isometric (resp.\ asymptotically isometric) if the connecting maps are isometries (resp.\ if 
\begin{equation}\label{eq:asymptotically_isometric}
    \lim_{n\gg m} \bigg( \inf_{a_m \in E_m \,:\,\norm{a_m}=1}\, \norm{\j nm a_m}\bigg)  =1 \, ).
\end{equation}
For an asymptotically isometric system, \jt-convergence of a bounded net $a\d$ is equivalent to convergence of $\j\oo n a_n$ in the limit space \cite{jdynamics}, i.e.,
\begin{equation}\label{eq:asymptotically_isometric2}
    a\d \text{ is \jt-convergent}\ \iff\ (\j\oo n a_n)_n \subset E_\oo \text{ is Cauchy},
\end{equation}
and the assumption of asymptotic isometricity cannot be removed altogether (see \cite[Footnote 1]{jdynamics}).

A \emph{split} inductive system of normed spaces is a triple $(E,j,s)$ consisting of a soft inductive system $(E,j)$ together with a linear contraction $s\d : E_\oo\to\C(E\d,j)$ that is a right-inverse to $\jlim : \C(\A\d,j)\to E_\oo$, i.e., satisfies
\begin{equation}
    \jlim_n s_n(a_\oo) = a_\oo,\qquad a_\oo \in E_\oo.
\end{equation}
In this case, one can define an equivalent family of connecting maps%
    \footnote{Let $(E,j)$ be a soft inductive system. Another family of connecting $\oj nm: E_n\to E_m$ is equivalent if \jt-convergence is equivalent to \ojt-convergence (defined as in \cref{eq:j_convergent} with the map $\oj nm$). In this case, $(E,\ojt\,)$ is also a soft inductive system with the same limit space.} 
by setting $\oj nm = s_n\circ \j\oo m$ (which makes sense for all $n,m$, i.e., without the condition $n\ge m$).

\begin{definition}
  Let $(E,j)$ be a soft inductive system of normed spaces, and let $N_0\subset N$ be a cofinal subset of the index set $(N,\leq)$. Then we call the net $\{E_n\}_{n\in N_0}$ along with the corresponding maps $j_{nm}$ a {\bf soft inductive subsystem}.
\end{definition}

One could give a more general definition of subsystems by generalizing how subnets of nets are defined, i.e., by allowing the index set of the subsystem to be a general directed set $(N',\le)$ together with a monotone cofinal map $f:N'\to N$.
We will, however, not need this level of generality here.

\begin{lemma}\label{lem:subsystem limit}
    Let $(E,j)$ be a soft inductive system of normed spaces and let $N_0\subset N$ be cofinal.
    Then, the limit space of the subsystem over $N_0$ is isometrically isomorphic with $E_\oo$.
\end{lemma}
\begin{proof}[Proof via the universal property.]
    Let us denote the connecting maps of the subsystem by $\oj nm$, $n>m\in N_0$, and the limit space by $\tilde E_\oo$.
    Every \jt-convergent net restricts to a \ojt-convergent net.
    In particular, for every $n\in N$, $a_n\in E_n$, we can view the basic net $\j\blob n a_n$ as an element of $\C(\tilde E\d,\ojt\,)$.
    Taking the \ojt-limit, we get maps $k_{\oo n}: E_n\to \tilde E_\oo$.
    We show that $(\tilde E_\oo,k_{\oo \blob})$ satisfies the universal property of the limit space of $(E,j)$.
    Let $\alpha_n:E_n\to F$, $n\in N$, be a net of contractions mapping \jt-convergent nets to Cauchy nets.
    Restricting this net to $N_0$ and invoking the universal property of the limit space of the subsystem, we get a linear contraction such that $\alpha_\oo:\tilde E_\oo\to F$ such that $\alpha_\oo(\ojlim_n a_n)=\lim_n \alpha_n(a_n)$, $a\d\in\C(\tilde E\d,\ojt\,)$.
    By construction, this implies $\alpha_\oo(\lim_n k_{\oo n} a_n) =\lim_n \alpha_n(a_n)$ for $a\d\in\C(E\d,j)$.
    Thus, the isomorphism exists by the uniqueness of the limit space; see Proposition~\ref{prop:univ_prop}.
\end{proof}

\begin{proof}[Direct proof]
    Let us denote the spaces of the subsystem by $\tilde E_n$, $n\in N_0$, the connecting maps by $\oj nm$, $n>m\in N_0$ and the limit space by $\tilde E_\oo$.
    The map $\C(E\d,j)\ni (a_n)_{n\in N}\mapsto (a_n)_{n\in N_0}\in \C(\tilde E\d,\ojt\,)$ is an isometry for the seminorms. 
    In particular, null nets are mapped to null nets.
    Thus, we get an isometry on the quotients $E_\oo =\C(E\d,j)/\C_0(E\d) \to \tilde E_\oo=\C(\tilde E\d,\ojt\,)$. 
    It remains to show that this is surjective.
    This is indeed the case because the limits of basic nets are dense in both limit spaces $E_\oo$.
\end{proof}

\begin{remark}\label{rem:summable to coherent}
   In \cite[Remark 2.1.4]{blackadar1997generalized} soft inductive systems are alluded to under the name of asymptotically coherent systems. 
   When the index is countable, it is sometimes useful to consider a stronger notion: 
    A soft inductive system $(E,j)$ which is indexed over $\mathbb{N}$ is \emph{summably soft} if  there exists some decreasing sequence $(\varepsilon_n)_n\in \ell^1(\mathbb{N})_+^1$ such that 
    \begin{align*}
        \|j_{nk}-j_{nm}\circ j_{mk}\|<\varepsilon_m,\ \forall\ n>m>k\geq 0.
    \end{align*}

Now, suppose $(E,j)$ is a summably soft inductive system of vector spaces. We define a strict system $(E,\ojt)$ from $(E,j)$ as follows: for all $n>m\geq 0$ set
\begin{align*}
    \ojt_{nm}&\coloneqq j_{n,n-1}\circ\hdots\circ j_{m+1,m},\\
    \ojt_{mn}&\coloneqq 0,\ \text{ and}\\
    \ojt_{nn}&\coloneqq \mathrm{id}_{E_n}.
\end{align*}
It follows from summability that for any $\varepsilon>0$ there exists an $M>0$ such that for all $n>m>M$, 
$$\|\ojt_{nm}-j_{nm}\|<\varepsilon.$$
In particular, $\|\ojt_{\infty n}-\j\oo n\|<\varepsilon$ for each $n>M$.  
Then 
\[\overline{\bigcup \ojt_{\infty n}(E_n)}=\overline{\bigcup \j\oo n(E_n)}\subset \Q(E\d),\]
and we denote both with $E'_\infty$.  

If $(E,j)$ is a summably soft system of finite-dimensional Banach spaces with countable index, we can use a compactness argument to pass to an asymptotically soft subsystem whose limit is isometrically isomorphic to $E'_\oo$ by Lemma~\ref{lem:subsystem limit}. 
\end{remark}

\subsection{Operator systems}\label{sec:op_sys}

We now turn to soft inductive limits of operator systems, by which we mean \emph{abstract} operator systems in the sense of \cite{paulsen_completely_2003}. Note that an operator system $\S$ is assumed to be unital, and we denote the unit by $1$ (or $1_{\S}$ if more context is needed).
The vector space of $\nu\times \nu$ matrices with entries in $\S$ will be referred to as a matrix amplification of $\S$ and is denoted $M_\nu(\S)$, and we use the shorthand $M_\nu=M_\nu(\CC)$.
Elements of $M_\nu(\S)$ will be denoted $[a_{ij}]_{ij}$ or just by $[a_{ij}]$ with $a_{ij}\in\S$. 
Of course, $M_\nu(\S)$ is the same as the (algebraic) tensor product $M_\nu \otimes \S$.
The standard matrix units of $M_\nu$ will be denoted $E_{ij}$, $i,j=1,\ldots,\nu$ so that we have $[a_{ij}] = \sum_{ij} a_{ij} \ox E_{ij}$.
If $\alpha:\S\to\T$ is a completely positive map between operator systems, we denote its matrix amplification by $\alpha\up\nu := \alpha\ox\id_\nu$ which acts by mapping $[a_{ij}]$ to the matrix $[\alpha(a_{ij})]$.
We always assume operator systems to be equipped with their order norm, i.e.,
\begin{equation}
    \norm{a}= \inf\set*{\lambda>0 \given
    \begin{bmatrix}
        \lambda1&a\\
        a^*&\lambda1
    \end{bmatrix}\ge 0}.
\end{equation}
On self-adjoint elements, this reduces to
\begin{equation}
    \norm{a} = \inf\set*{\lambda>0 \given -\lambda1 \le a \le \lambda1}.
\end{equation}
In particular, operator systems are normed spaces, and unital completely positive (ucp) maps are linear contractions, cf.\ \cite[Prop.\ 3.6]{paulsen_completely_2003}.
We do not require operator systems to be complete with respect to the order norm.

\begin{definition}
    A {\bf soft inductive system in the operator system category} $(\S,j)$ over a directed set $(N,\le)$ is a soft inductive system of normed spaces such that each $\S_n$ is an operator system and each $\j nm$ is a ucp map.  The system is called (completely) isometric when all the $j_{nm}$ are. 
\end{definition}

We let $(\S,j)$ be a soft inductive system of operator systems. Then the uniformly bounded nets $\nets(\S\d)$ form an operator system, namely the direct product $\prod_{n\in N}\S_n$:
The adjoint and the vector space structure are defined $n$-wise, the order unit is $ 1\d$, and the positive cone is $\nets(\S\d)^+ =\set{a\d\in\nets(\S\d)\given a_n\ge0 \ \forall n}$.
On the matrix amplifications, the positive cones are given through the natural isomorphism 
\begin{equation}\label{eq:nets_matrix_amp}
    M_\nu(\nets(\S\d)) \cong \nets(M_\nu(\S)\d).
\end{equation}
It immediately follows from $(\j nm a_m)^*=\j nm(a_m^*)$ that, for any \jt-convergent net $a\d$, also the adjoint $a\d^*=(a_n^*)_{n\in N}$ is \jt-convergent.
Furthermore, by unitality of the connecting maps, the net of order units $ 1\d$ is \jt-convergent.
Thus, the \jt-convergent nets $\C(\S\d,j)$ are a unital self-adjoint subspace of the operator system $\nets(\S\d)$ and, hence, an operator subsystem of $\nets(\S\d)$.

\begin{lemma}\label{lem:matrix_amp_lemma}
    \begin{enumerate}[(1)]
        \item 
            The matrix amplifications $M_\nu(\S_n)$ together with the connecting maps $\jnu nm$ also form a soft inductive system of operator systems $(M_\nu(\S),j\up\nu)$.
            A net $[a_{ij,\blob}]\in \nets(M_\nu(\S)\d)$ is \jtnu-convergent if and only if each $a_{ij,\blob}$, $i,j=1,\ldots,\nu$, is \jt-convergent.
        \item 
            Let $[a_{ij,\blob}]$ be \jtnu-convergent (and positive) and let $A$ be any complex $\mu\times\nu$-matrix.
            Then the net $A[a_{jk,\blob}]A^*= \sum_{ijkl} A_{ij}\Bar A_{lk} \, a_{jk,\blob}\ox E_{il}$ is \jtmu-convergent (and positive).
        \item
            For any (positive) $A\in M_\nu$ and (positive) $a\d\in\C(\S\d,j)$ the net $A\otimes a\d = [A_{jk}a\d]$ is \jtnu-convergent (and positive).
    \end{enumerate} 
\end{lemma}

\begin{proof}
    (1):
    The "if" part follows from the inequality
    \[
        \norm{(\jnu nl - \jnu nm\jnu ml)[a_{ij,l}]} = \norm{ \sum_{ij} (\j nl -\j nm \j ml)(a_{ij,l})\ox E_{ij}} \le \sum_{ij} \norm{(\j nl -\j nm\j ml)a_{ij,l}} \xrightarrow{n\gg m}0.
    \] 
    For the converse, we start with basic sequences. For these the claim is immediate because $\jnu \blob m [a_{ij,m}]=[\j \blob m a_{ij,m}]$.
    Now let $[a_{ij,\blob}]$ be \jtnu-convergent. Then \jt-convergence of $a_{ij,\blob}$ follows from the inequality $\norm{[a_{ij,n}]}\ge \max_{ij} \norm{a_{ij,n}}$.

    (2): Positivity of $A[a_{jk,n}]A^*$ follows from $\S_n$ being matrix ordered. 
    Since this holds for all $n$, we know that $A[a_{jk,\blob}]A^*$ is positive.
    \jt-convergence is readily checked:
    \[
        \norm{(\jnu nl-\jnu nm\jnu ml)(A[a_{jk,l}]A^*)} \le \sum_{ijkr} \abs{A_{ij}A_{kr}} \norm{(\j nl-\j nm\j ml)a_{jk} } \xrightarrow{n\gg m}0.
    \] 

    (3): This is a special case of (2).
\end{proof}

This Lemma shows that the operator systems structure of $\C(\S\d,j)$, inherited from $\nets(\S\d)$, is compatible with the notion of \jt-convergence.
In particular, the isomorphism \eqref{eq:nets_matrix_amp} yields $M_\nu(\C(\S\d,j))\cong \C(M_\nu(\S)\d,j\up\nu)$. 
It follows from their construction that 
\[  
    \S'_\oo:=\overline{\bigcup_n j_{\oo n}(\S_n)}\subset \Q(\S\d) \quad\text{and}\quad \S_\oo = \C(\S\d,j)/\C_0(\S\d)
\] 
agree as $^*$-Banach spaces. Moreover, since $M_\nu(\C(\S\d,j))\cong \C(M_\nu(\S)\d,j\up\nu)$, we can also identify
\begin{align*}
    \textstyle M_\nu(\S_\infty)=M_\nu(\S'_\infty)=M_\nu\left(\overline{\textstyle{\bigcup_n} j_{\oo n}(\S_n)}\right)\subset M_\nu(\Q(\S\d))\end{align*}
    with
\begin{align*}
\textstyle M_\nu(\S)_\oo=M_\nu(\S)'_\oo=\overline{\textstyle{\bigcup_n} j^{(\nu)}_{\oo n}(M_\nu(\S_n))}\subset \Q(M_\nu(\S)\d).\end{align*}

We have to equip the limit space $\S_\oo = \C(\S\d,j)/\C_0(\S\d)$ with an operator system structure. One approach is to consider the $\S_n$'s as concrete operator systems: for each $\S_n$ in the system, the Choi-Effros theorem allows us to concretely represent it as an operator subsystem of some unital $\Cstar$-algebra $\A_n$. 
Then $\C(\S\d,j)$ is an operator subsystem of $\nets(\S\d)$, which is an operator subsystem of the unital $\Cstar$-algebra $\prod_n \A_n$, and $\C_0(\S\d)$ is the intersection of $\C(\S\d,j)$ (or $\nets(\S\d)$) with the ideal $\bigoplus_n \A_n$ of null convergent nets. 
Then $\S_\oo$ and $\Q(\S\d)$ inherit an operator system structure by identifying them with the images of $\C(\S\d,j)$, resp.\ $\nets(\S\d)$, under the quotient map $\prod_n \A_n \to \prod_n \A_n/\bigoplus_n \A_n$. 
The resulting operator system structure is independent of the choice of the algebras $\A_n$ since we will always have $\C_0(\S\d)=\C(\S\d)\cap\C_0(\A\d)$.

However, one may also define an operator system on $\S_\oo$ and $\Q(\S\d)$ abstractly as operator system quotients as introduced and studied in \cite{kavruk2013quotients}. We briefly summarize the important bits:

The quotient is naturally equipped with the adjoint operation $(\jlim_n a_n)^* =\jlim_n a_n^*$ and contains the special element $1_\oo =\jlim_n a_n$. 
One can also define a system of matrix-compatible positive cones $D_\nu$ via the image of the positive cones of $M_\nu(\C(\S\d,j))$ under the canonical projection $\jnulim$. It is, however, not immediate that these define an operator system structure (see the discussion following Proposition~\ref{prop:Kav+5.12}).
Instead of constructing an operator system structure on $\S_\oo$ by hand, we rely on \cite{kavruk2013quotients} where quotients of operators systems have been introduced and studied.
We briefly summarize the important bits:

For a quotient $\TT/\J$ of an operator system, $\TT$ and (non-unital) subspace $\J$ to make sense (in the category of operator systems), one needs that $\J$ is a so-called \emph{kernel}.
Among other equivalent definitions (see \cite[Prop. 3.1]{kavruk2013quotients}) this means that there is a ucp map $\phi$ with domain $\TT$ so that $\J=\ker \phi$.
Every kernel $\J$ is an \emph{order ideal} which means that $a\in\TT$ and $b\in\J$ so that $0\le a\le b$ implies $a\in\J$.
The matrix order of the quotient is defined by
\begin{equation}\label{Dv}
    M_\nu(\TT/\J)^+ = \set{ a+\J \in M_\nu(\TT) \given \forall_{\eps>0}\exists_{b\in\J} : \eps 1 + a+b \ge0 }
\end{equation}
and the quotient operator system is characterized by the following universal property:

\begin{proposition}[Universal property of the quotient operator system {\cite[Prop.~3.6]{kavruk2013quotients}}]\label{prop:upropunital}
    Every unital completely positive map $\varphi:\T_1\to\T_2$ that vanishes on $\J$, $\J\subset\ker(\varphi)$, factors through the quotient.
    That is, the map $\hat\varphi:\T_1/\J\to\T_2$ with $\hat\varphi(x+\J) = \varphi(x)$ is well-defined, unital and completely positive.
    
    Let $\mc R$ be an operator system and $\psi:\T_1\to\mc R$ a unital completely positive map.
    If it holds that for every unital completely positive map $\varphi:\T_1\to\T_2$ with $\J\subset \ker(\varphi)$ there exists a unique unital completely positive map $\hat\varphi:\mc R\to\T_2$ such that $\varphi = \hat\varphi\circ\psi$ then there exists a unique complete order isomorphism $\gamma:\mc R\to\T_1/\J$ such that $\gamma\circ\psi$ is the canonical quotient map.
\end{proposition}

\begin{corollary}\label{cor:upropcont}
    Every contractive completely positive map $\varphi:\T_1\to\T_2$ that vanishes on $\J$, $\J\subset\ker(\varphi)$, factors through the quotient.
    That is, the map $\hat\varphi:\T_1/\J\to\T_2$ with $\hat\varphi(x+\J) = \varphi(x)$ is well-defined, contractive and completely positive.
\end{corollary}
\begin{proof}
        Let $\varphi:\T_1\to\T_2$ be a contractive completely positive map. Then we can view $\varphi$ as a unital completely positive map $\T_1\to\varphi(\T_1) =: \T$ where we choose $\varphi(1_{\T_1}) =: 1_{\T}$ as the order unit of $\T$.
        The universal property Proposition~\ref{prop:upropunital} then gives existence of a unique unital completely positive map $\hat\varphi:\T_1/\J\to\T$ necessarily given by $\hat\varphi(x+\J) = \varphi(x)$.
        Note that, since $\varphi$ is contractive, it holds that $\norm{1_{\T}}_{\T_2} \le \norm{1_{\T_1}}_{\T_1} = 1 = \norm{1_{\T}}_\T$, so the embedding $\T\hookrightarrow\T_2$ is contractive.
        Thus, viewing $\hat\varphi$ as a map $\T_1/\J\to\T_2$ gives a completely positive contraction.
\end{proof}

We now return to soft inductive systems of operator systems. Let us assume that each $\S_n$ is concretely realized as a unital self-adjoint subspace of a unital $\Cstar$-algebra $\A_n$.
Since $\C_0(\S\d)$ is the intersection of $\C(\S\d,j)$ (or $\nets(\S\d)$) with the $\Cstar$-ideal $\bigoplus_n \A_n$, it is a kernel in the sense of \cite[Definition 3.2]{kavruk2013quotients}. Hence, it has the universal operator system structure with matrix order as in \eqref{Dv}. In particular, its operator system structure is independent of the choice of the $A_n$.)
Below, we give two intrinsic proofs that $\C_0(\S\d)$ is a kernel.

\begin{lemma}
    \begin{enumerate}[(1)]
        \item\label{it:order_ideal}
            $\C_0(\S\d)$ is a self-adjoint closed order ideal in $\C(\S\d,j)$.
        \item\label{it:order_unit_norm}
            The order unit induced norm on the quotient $\C(\S\d,j)/\C_0(\S\d)$ coincides with the Banach space quotient norm, i.e.,
            \begin{equation}\label{eq:order_unit_norm}
                \inf_{b\in \C_0(\S\d)} \norm{a\d+b\d}_\nets =\inf \set{\lambda>0 \given \jlim_na_n\pm \lambda 1_\oo \in D_1} \quad \forall a\d=a\d^*\in\C(\S\d,j).
            \end{equation}
        \item\label{it:kernel}
            $\C_0(\S\d)$ is a kernel in the sense of \cite{kavruk2013quotients}. 
    \end{enumerate} 
\end{lemma}

\begin{proof}
    \ref{it:order_ideal}: Let $0\le a\d\le b\d$ with $b\d\in\C_0(\S\d)$. Then $\norm{a_n}\le \norm{b_n}$ for all $n$ which implies that $\seminorm{a\d}=\lim_n \norm{a_n} \le\lim_n \norm{b_n}=0$ and hence $a\d\in\C_0(\S\d)$.

    \ref{it:order_unit_norm}:
    Now let $a\d=a\d^*\in\C(\S\d,j)$. By \cite[Eq.~(15)]{jdynamics}, the left-hand side of \eqref{eq:order_unit_norm} is equal to $\seminorm{a\d}$.
    We can WLOG assume that $\norm{a_n}\le\seminorm{a\d}$ \cite[Prop.~5.(1)]{jdynamics} which implies that $-\seminorm{a\d}1\d\le a\d \le \seminorm{a\d}1\d$ because each $\S_n$ is an order unit space.
    Therefore $\seminorm{a\d}$ is larger than the infimum on the right hand-side of \eqref{eq:order_unit_norm}.
    Now let $\lambda>0$ be such that $-\lambda 1_\oo \le \jlim_n a_n\le \lambda 1_\oo$. 
    This means that there is a $b\d=b\d^*\in\C(\S\d,j)$ with the same \jt-limit as $a\d$ and $\check e\d=\check e\d^*,\hat e\d=\hat e\d^*\in\C(S,j)$ with $\jlim_n \check e_n=\jlim_n\hat e_n = 1_\oo$, such that $-\lambda \check e\d \le b\d\le \lambda\hat e\d$.
    Taking seminorms of this implies that $\seminorm{b\d}\le \lambda\seminorm{\hat e\d}$ but since the seminorm of a \jt-convergent is just the norm of its limit, this inequality simply gives $\norm{\jlim_n a_n} \le \lambda$.
    This finishes the proof of \eqref{eq:order_unit_norm}.

    \ref{it:kernel}: According to \cite[Lem. 3.3]{kavruk2013quotients}, \ref{it:order_ideal} and \ref{it:order_unit_norm} show that a necessary and sufficient condition for being a kernel is met.
\end{proof}

We give a third proof that $\C_0(\S\d)$ is a kernel:

\begin{proof}[Direct proof that $\C_0(\S\d)\subset\C(\S\d,j)$ is a kernel]
    By \cite[Prop. 3.1]{kavruk2013quotients} it suffices to find a family of states $\Omega$ on $\C(\S\d,j)$ such that $\C_0(\S\d)=\bigcap_{\psi\in\Omega}\ker\psi$.
    Let $\omega$ be an ultrafilter on $N$ that is compatible with the order structure in the sense that it contains all cofinite sets in $N$.%
    \footnote{A subset $M$ of a directed set $(N,\le)$ is cofinal if every element of $N$ is dominated by some element of $M$. To see that such a filter exists, consider the filter of all cofinite sets and pick any ultrafilter $\omega$ containing it. These filters are free (or non-principal) if and only if $N$ does not contain a maximal element.}
    Let $\psi\d$ be \emph{any} net of states (i.e., $\psi_n$ is a state on $\S_n$ for all $n$). 
    Then one obtains a state $\psi_{\omega}$ on $\C(\S\d,j)$ which maps a \jt-convergent net $a\d$ to $\lim_{n\to\omega} \psi_n(a_n)$. 
    We set $\Omega = \{\psi_\omega : \psi_n\in S_n,\ n\in N, \ \omega\}$, where "$\omega$" ranges over all order-compatible ultrafilters on $N$.
    It is clear that any null net $a\d$ is in the kernel of any such state, i.e., $\C_0(\S\d)\subset \bigcap_{\psi\in\Omega}\ker \psi$. 
    Conversely, assume that a, without loss of generality, self-adjoint net $a\d=a\d^*$ is in the kernel of all $\psi\in\Omega$.
    Pick a net of states $\psi\d$ such that $\norm{a_n}=\abs{\psi_n(a_n)}$.
    Then $\lim_{n\to\omega} \norm{a_n} = \lim_{n\to\omega} |\psi_n(a_n)| = |\psi_\omega(a\d)|=0$ for all order-compatible ultrafilters $\omega$ on $N$ and, hence, $\lim_n\norm{a_n}=0$, i.e., $a\d\in\C_0(\S\d)$.
    Therefore, $\C_0(\S\d)=\bigcap_{\psi\in\Omega}\ker\psi$.
\end{proof}

Knowing that $\C_0(\S\d)$ is a kernel in $\C(\S\d,j)$, we can define the limit space as the operator system quotient $\S_\oo\coloneqq \C(\S\d,j)/\C_0(\S\d)$ as in \cite[Prop. 3.4]{kavruk2013quotients}.
This means that the order unit is $ 1_\oo$ and the positive cones are given by
\begin{multline}\label{eq:cones_for_limit}
    M_\nu(\S_\oo)^+ \\= \jnulim \paren*{\set[\big]{ a\d\in\C(M_\nu(\S),j\up\nu) \given \forall_{\eps>0} \exists_{b\d\in\C_0(M_\nu(\S),j\up\nu)} : \eps 1 + a_{\blob}+ b\d \ge0}}.
\end{multline}
The canonical projection $\jlim :\C(\S\d,j)\to \S_\infty$ is a ucp map, and the quotient norm on the limit space (which is the same as the seminorm of any representative) coincides with the norm induced by the order unit.
The same is true for all matrix amplifications.

It is then natural to ask whether the universal operator system structure on $\S_\oo$ with positive cones as in \cref{eq:cones_for_limit} agrees with the operator system structure $\S_\oo$ inherits as a subspace of the $\Cstar$-algebra  $\nets(\S\d)/\C_0(\S\d)\equiv \prod_n \A_n/\bigoplus_n \A_n$ as above.
By  \cite[Proposition 3.6]{kavruk2013quotients}, the induced injective linear map from the abstract operator system $\S_\infty=\C(\S\d,j)/\C_0(\S\d)$ (resp.\ $\Q(\S\d)$) into the $\Cstar$-quotient $\nets(\S\d)/\C_0(\S\d)$ is ucp. However, in general, such a map need not be a complete order embedding (see the remarks following \cite[Proposition 3.6]{kavruk2013quotients}). In our case, it will follow from \cite[Proposition 5.12]{kavruk2013quotients} that the map is a complete order embedding. 

\begin{proposition}{\cite[Proposition 5.12]{kavruk2013quotients}}\label{prop:Kav+5.12}
    Let $\mathcal{I}$ be an ideal of a $\Cstar$-algebra $\A$, $\{e_\alpha\}_\alpha$ be a quasi-central approximate unit for $\mathcal{I}$ and $\S\subset \A$ be an operator system. If $e_\alpha s\in \J=\S\cap \mathcal{I}$ for every $s\in \S$ and for every $\alpha$, then the induced map $\S/\J\to \A/\mathcal{I}$ is a (unital) complete order isomorphism. The operator space and operator system quotients $\S/\J$ are completely isometric and $\J$ is a completely biproximinal kernel in $\S$. 
\end{proposition}

For more on operator space quotients and biproximinal kernels, see \cite[Section 4]{kavruk2013quotients}. In particular, a completely biproximinal kernel is completely order proximinal (\cite[Definition 3.5]{kavruk2013quotients}): A kernel $\J$ of an operator system $\TT$ is said to be completely order proximinal if the matrix order constructed in \eqref{eq:cones_for_limit} has the property
\begin{equation}\label{eq:proximinal}
    a+\J \in M_\nu(\TT/\J)^+ \iff \exists_{b \in \J} : a+b \ge0.
\end{equation}
  
\begin{proposition}\label{prop: limit as quotient}
    Let $(\S,j)$ be a soft inductive system of operator systems. Then the operator space and operator system quotients $\C(\S\d,j)/\C_0(\S\d)$ (resp.\ $\nets(\S\d)/\C_0(\S\d)$) are completely isometric and $\C_0(\S\d)$ is a completely biproximinal kernel in both $\C(\S\d,j)$ and $\nets(\S\d)$. Moreover, if $\{\A_n\}$ is a family of $\Cstar$-algebras such that $\S_n$ is concretely represented as an operator subsystem of $\A_n$ for each $n$, then the induced embedding $\S_\oo \to \nets(\S\d)/\C_0(\S\d)$ is a complete order isomorphism onto $\C(\S\d,j)/\C_0(\S\d)$, identified as a subspace of  $\nets(\S\d)/\C_0(\S\d)$. 
\end{proposition}

\begin{proof}
   Let $\{\A_n\}$ be a family of $\Cstar$-algebras such that $\S_n$ is concretely represented as an operator subsystem of $\A_n$ for each $n$. Using the quasi-central approximate identity $\{e^{(n)}_\blob\}_n$ of $\bigoplus_n \A_n$ where $e^{(n)}_k= 1_{\A_k}$ if $k\leq n$ and $0$ otherwise, $\C_0(\S\d)$ satisfies these conditions in Proposition~\ref{prop:Kav+5.12} as a kernel in both $\C(\S\d,j)$ and in $\nets(\S\d)$. The conclusion follows. 
\end{proof}

\begin{corollary}\label{cor:proximinal}
    Let $(\S,j)$ be a soft inductive system of operator systems and $\{\A_n\}$ a family of $\Cstar$-algebras such that $\S_n$ is concretely represented as an operator subsystem of $\A_n$ for each $n$. Then the universal operator system structure on $\S_\oo$ agrees with the operator system structure on $\S_\oo$ it inherits as a subspace of $\nets(\S\d)/\C_0(\S\d)\equiv \prod_n \A_n/\bigoplus_n \A_n$. Moreover, for all $\nu\geq 1$ we have 
    \begin{align*}
         M_\nu(\S_\oo)^+ 
         :\!\!&= \jnulim \paren*{\set[\big]{ a\d\in\C(M_\nu(\S),j\up\nu) \given \forall_{\eps>0} \exists_{b\d\in\C_0(M_\nu(\S),j\up\nu)} : \eps 1 + a_{\blob}+ b\d \ge0}}\\
         &= \jnulim \paren*{\set[\big]{ a\d\in\C(M_\nu(\S),j\up\nu) \given  \exists_{b\d\in\C_0(M_\nu(\S),j\up\nu)} : a_{\blob}+ b\d \ge0}}.
    \end{align*}
    In particular,  for each $a\in M_\nu(\S_\oo)^+$, there exists an $a\d\in \C(M_\nu(\S),j\up\nu)^+$ such that $\jnulim(a\d)=a$, i.e., positive elements lift to positive elements. 
\end{corollary}

The following result is the operator system version of \cite[Prop. 22]{jdynamics}:
\begin{proposition}\label{prop:jjconvlim}
    Let $(\S, j)$ and $(\T,k)$ be soft inductive systems of operator systems over the same directed set.
    Let $(\varphi_n:\S_n\to\T_n)_{n\in N}$ be a net of completely positive contractions taking \jt-convergent to $k$-convergent nets.
    Then there is a completely positive contraction $\varphi_\infty:\S_\infty\to\T_\infty$ such that
    \begin{equation}
        \varphi_\infty\Big(\jlim_n a_n\Big) = \klim_n \varphi_n(a_n).
    \end{equation}
\end{proposition}
\begin{proof}
The map $\varphi\d$ defines a completely positive contraction $\C(\S\d,j)\to\C(\T,k)$ by $a\d\mapsto \varphi\d(a\d)=(\varphi_n(a_n))_{n\in N}$. Note, that complete positivity follows from the identification $M_\nu(\C(\S\d,j))\cong \C(M_\nu(\S)\d,j\up\nu)$.
    Composing with the quotient map $\klim$ and invoking the universal property (Corollary~\ref{cor:upropcont}) for the resulting cp contraction $\C(\S\d,j)\to\T_\infty$ proves the statement.
\end{proof}

By the same proof, one can show a more general version of Proposition~\ref{prop:jjconvlim} where the index set $N$ of $(\S,j)$ and $M$ of $(\T,k)$ need not be the same but instead are connected via a monotone map $f:N\to M$.
For instance, the index set of $(\S,j)$ might be a cofinal subset $N\subset N$ of the index set $M$ of $(\T,j)$ with $f:N\hookrightarrow M$ the inclusion map.
The universal property of the limit operator system is a direct consequence of Proposition~\ref{prop:jjconvlim} (cp.\ \cite[Prop.~13 \& 14]{jdynamics}):

\begin{proposition}[Universal property of the limit operator system]\label{prop:uproplimit}
    Let $(\S,j)$ be a soft inductive system of operator systems, and let $\T$ be another operator system.
    Let $\varphi\d$ be a net of completely positive contractions $\S_n \to \T$ which maps $j$-convergent nets to Cauchy nets in $\T$, then there is a (unique) completely positive contraction $\varphi_\infty: \S_\infty \to \overline\T$ into the completion of $\T$ such that $\varphi_\infty(\jlim a_n) = \lim_n \varphi_n(a_n)$.

    Let $\tilde \S_\oo$ be a norm-complete operator system and $\oj n\oo:\S_n\to\tilde\S_\oo$ be a net of unital completely positive maps mapping \jt-convergent to Cauchy nets.
    For $a\d\in\C(\S\d,j)$, set $\ojlim_n a_n = \lim_n \oj\oo na_n\in \tilde\S_\oo$.
    Suppose that for every operator system $\T$ and every net $\varphi_n : \S_n \to \T$ of unital completely maps taking \jt-convergent to Cauchy nets, there exists a map $\varphi_\oo : \tilde\S_\oo \to \T$ such that $\varphi_\oo(\jlim_na_n) = \lim_n \varphi_n(a_n)$, $a\d\in\C(\S\d,j)$.
    Then there is a complete order isomorphism $\psi : \tilde\S_\oo \to \S_\oo$ such that $\jlim_n a_n = \psi(\ojlim_na_n)$, $a\d\in\C(\S\d,j)$.
\end{proposition}

\begin{proof}
    The first part is a straightforward application of Proposition~\ref{prop:jjconvlim} to the constant inductive system $(\T, \id)$.
    As for the second part, let $\pi:\C(\S\d,j)\to\T$ denote the map $a\d\to \lim_n \varphi_n(a_n)$. 
    It follows right from the definition that $\pi$ is contractive.
    The identification $M_\nu(\C(\S\d,j))\cong \C(M_\nu(\S)\d,j\up\nu)$ and the cones of $M_\nu(\mc R)$ being closed imply that $\pi$ is completely positive.
    It is also clear that $\pi$ vanishes on $\C_0(\S\d)$.
    By the universal property of the quotient operator system, there exists a completely positive contraction $\psi:\S_\infty\to\mc R$ such that $\pi = \psi\circ\jlim$.
    The last bit is then the basic reasoning for universal properties.
\end{proof}

From the universal property, we conclude that Lemma~\ref{lem:subsystem limit} holds in the operator system category:

\begin{corollary}\label{cor:coi limits}
    The limit of a soft inductive system of operator systems is completely order isomorphic to the limit of a subsystem. 
\end{corollary} 

\begin{proof}
    Since we established the universal property of the soft inductive limit in the operator system category (cp.\ Proposition~\ref{prop:upropunital}), the proof of Lemma~\ref{lem:subsystem limit} (which is the same statement in the category of normed spaces) now applies in the operator system category.
\end{proof}



\subsection{\texorpdfstring{$\Cstar$}{C*}-algebras}\label{sec:cstar}

In the following, all $\Cstar$-algebras are assumed to be unital if not stated otherwise.
We recall the following definition from \cite{jdynamics}:

\begin{definition}
    A {\bf soft inductive system of $\Cstar$-algebras} over a directed set $(N,\le)$ is a soft inductive system of normed spaces $(\A,j)$ where each $\A_n$ is a $\Cstar$-algebra and where the connecting $\j nm$ are ucp maps that are {\bf asymptotically multiplicative}
    \begin{equation}\label{eq:asymptotic_multiplicativity}
    \lim_{n\gg m}\,\norm{\j nm((\j ml a_l)(\j ml b_l))- (\j nl a_l)(\j nl b_l)} =0\quad\forall l\in N,\, a_l,b_l\in \A_l.
    \end{equation}
\end{definition}

In particular, a soft inductive system of $\Cstar$-algebras is a soft inductive system of operator systems.
Therefore, the net of units $ 1\d$ is \jt-convergent and a net $a\d$ is \jt-convergent if and only if its adjoint $a\d^*$ is.
The assumption of asymptotic multiplicativity further guarantees that products of \jt-convergent nets are \jt-convergent. In fact, this is equivalent to asymptotic multiplicativity (see Proposition~\ref{prop:products} below).
Therefore $\C(\A\d,j)$ is a unital $\Cstar$-subalgebra of $\nets(\A\d)$ and $\C_0(\A\d)$ is a closed two-sided ideal.
Null nets form a closed two-sided ideal.
Therefore, the limit space is naturally a unital $\Cstar$-algebra.
It can be shown that the continuous functional calculus $f(a\d)$ of a \jt-convergent net $a\d$ of normal elements is again \jt-convergent and that $\Sp(\jlim_n a_n) \subset \bigcap_{k}\bigcup_{n>k} \Sp(a_n)$.

Among soft inductive systems of operator systems, soft inductive systems of $\Cstar$-algebras are characterized as follows:

\begin{proposition}\label{prop:products}
    Let $(\A,j)$ be a soft inductive system of operator systems where each $\A_n$ is a $\Cstar$-algebra.
    Consider the following statements:
    \begin{enumerate}[(a)]
        \item\label{it:products1}
            $n$-wise products of \jt-convergent nets are \jt-convergent.
        \item\label{it:products2}
            the connecting maps $\{\j nm\}$ are asymptotically multiplicative, see \eqref{eq:asymptotic_multiplicativity}. By definition, this means that $(\A,j)$ is a soft inductive system of $\Cstar$-algebras.
        \item \label{it:products3}
            The limit space $\A_\oo$, naturally embedded into the quotient $\Cstar$-algebra $\nets(\A\d)/\C_0(\A\d)$, is closed under products.
    \end{enumerate}
    Then (a) $\iff$ (b) $\implies$ (c). If $(\A,j)$ is asymptotically isometric, then they're all equivalent.
\end{proposition}
\begin{proof}
    \ref{it:products1} $\Rightarrow$ \ref{it:products2} $\Rightarrow$ \ref{it:products3} are clear.
    Note that \ref{it:products1}, i.e., asymptotic mulitplicativity \eqref{eq:asymptotic_multiplicativity}, is equivalent to $n$-wise products of basic nets being $j$-convergent. 
    Since basic nets are dense, we have \ref{it:products2} $\Rightarrow$ \ref{it:products1}.
    If we assume asymptotic isometricity (see \eqref{eq:asymptotically_isometric}), then \ref{it:products3} $\Rightarrow$ \ref{it:products2} follows from \eqref{eq:asymptotically_isometric}.
\end{proof}

Soft inductive systems of $\Cstar$-algebras were originally studied in the countably indexed and strict setting by Blackadar and Kirchberg in \cite{blackadar1997generalized}. When the $\Cstar$-algebras appearing in the system are finite-dimensional, the system is called \emph{NF}, and Blackadar and Kirchberg show in \cite{blackadar1997generalized} that a separable $\Cstar$-algebra is nuclear and quasidiagonal if and only if it is isomorphic to the limit of an NF system. Such $\Cstar$ algebras are called \emph{NF}. (Recall that a $\Cstar$-algebra is \emph{nuclear} provided its tensor product with any other $\Cstar$-algebra admits a unique $\Cstar$-norm (see  \cref{sec:nuclear} for the completely positive approximation property characterization), and a $\Cstar$-algebra is \emph{quasidiagonal} (QD) if there exists a net of cpc maps $\psi_n:\A\to M_{\nu(n)}$ that is pointwise asymptotically multiplicative ($\|\psi_n(a)\psi_n(b)-\psi_n(ab)\|\to 0$ for all $a,b\in\A$) and pointwise asymptotically isometric ($\|\psi_n(a)\|\to \|a\|$  for all $a\in\A$). 

\begin{corollary}
    Let $(\A,j)$ be a soft inductive system of $\Cstar$-algebras where each $\A_n$ is a finite-dimensional $\Cstar$-algebra. Then the limit $\A_\oo$ is nuclear and quasidiagonal. In particular, if $\A_\oo$ is separable, then it is NF.
\end{corollary}

\begin{proof}
    We identify $\A_\oo$ with $\overline{\bigcup_n j_{\oo n}(\A_n)}\subset \Q(\A\d)$. That $\A_\infty$ is nuclear follows from Ozawa and Sato's One-Way-CPAP (which can be found implicitly in \cite{Ozawa2002} (via \cite{KS03}); see \cite[Theorem 5.1]{Sato2019} for the explicit statement and proof) using the same argument as in \cite[Corollary 3.2]{C23} (which follows \cite[Theorem 2.13]{CW1}).\footnote{Even though the systems here are not necessarily strict, the map $\Phi$ from \cite[Corollary 3.2]{C23} is still just the inclusion of $\A_\oo$ into $\ell^\infty (N,\A_\oo)/c_0(N, \A_\oo)$ as constant sequences.} 
    
    Now, since $\A_\oo$ is nuclear, it has the local lifting property, which means in particular that for any finite-dimensional operator subsystem $\S$ of $\A$, there exists a ucp map $\psi_{\S}:\S\to \nets(\A\d)$ so that $[\psi_{\S}(a)]=a$ for all $a\in \S$. Composing these with the projections $\nets(\A\d)\to \A_n$, we obtain ucp maps $\psi_{\S,n}:\S\to \A_n\subset M_{\nu(n)}$ (where $M_{\nu(n)}$ is a full matrix algebra containing $\A_n$). By Arveson's extension theorem, these extend to ucp maps $\A\to \A_n\subset M_{\nu(n)}$ which we also denote with $\psi_{\S,n}$. This gives us our desired net of pointwise asymptotically multiplicative and pointwise asymptotically isometric cpc maps. 
\end{proof}

\section{Nuclearity}\label{sec:nuclear}
Originally, a $\Cstar$-algebra was said to be nuclear when its algebraic tensor product with any other $\Cstar$-algebra admits a unique $\Cstar$-norm, or equivalently, that the maximum and minimum tensor product norms agree.
In the 70's, Choi and Effros and (independently) Kirchberg characterized nuclearity for $\Cstar$-algebras in terms of a completely positive approximation property:  

\begin{theorem}[Choi--Effros/Kirchberg]\label{thm:CEK}
    A $\Cstar$-algebra $\A$ is nuclear if there exists a net of cpc maps $\A\xrightarrow{\psi_n}M_{\nu_n}\xrightarrow{\varphi_n}\A$ such that $\varphi_n\circ\psi_n$ converge pointwise in norm to the identity map on $\A$. When $\A$ is separable, the index can be taken to be $\mathbb{N}$. 
\end{theorem}

Following this, in \cite{Kir95}, Kirchberg defined nuclearity for operator systems using a similar property: a separable operator system $\S$ is \emph{nuclear} if there are completely positive contractions $\varphi_n:\S\to M_{\nu_n}$ and $\psi_n:M_{\nu_n}\to\S$ such that $\varphi_n\circ\psi_n\to \id_\S$ in the point-norm topology.\footnote{The topology on $\S$ is inherited from its order norm, or equivalently from its enveloping $\Cstar$-algebra $\Cstar_{\text{min}}(\S)$. Note that $\S$ need not be complete for this definition.}

\begin{remark}
    In \cite[Theorem 3.1]{han2011approximation}, Han and Paulsen show that Kirchberg's definition of nuclearity agrees with (max,min)-nuclearity of \cite{kavruk2011tensor}, i.e., an operator system $\S$ is nuclear if and only if its minimal tensor product with an arbitrary operator system agrees with the maximal tensor product.
\end{remark} 

The system of completely positive approximations from Theorem~\ref{thm:CEK} induces a soft inductive system of operator systems $(\S,j)$ with $\S_n\coloneqq M_{\nu_n}$ and $\j nm \coloneqq \psi_n\circ\varphi_m$, which we call the \emph{induced soft system}.   
Indeed, one has for each $l\in N$ and $a_l\in \S_l$
\[
    \norm{(\j nl-\j nm\j ml)a_l} \leq \norm{\psi_n} \norm{(\id_\S - \varphi_m\psi_m)\varphi_l(a_l)}\xrightarrow{n\gg m}0.
\] 
Note that, by construction, the induced soft system is split, where the right inverse to $\jlim$ is $\varphi_{\d}$.

\begin{theorem}\label{thm:nuclearity_iff}
    Let $\S$ be a nuclear operator system, then $\S$ is a soft inductive limit of matrix algebras, i.e., there is a soft inductive system $(\S,j)$ with each $\S_n$ being a matrix algebra so that $\S$ is its limit space.
    In fact, the soft inductive system can be chosen to be split. 
    If $\S$ is separable, then the directed system can be chosen to be $N=\NN$, and the soft inductive system can be arranged to be strict. 
\end{theorem}

\begin{proof}
   Given a net of completely positive contractions $\varphi_n:\S\to M_{\nu_n}$ and $\psi_n:M_{\nu_n}\to\S$ such that $\varphi_n\circ\psi_n$ converges point norm to $\id_\S$, we form the induced soft inductive system 
   $(\S,j)$ as above. The limit space $\S_\oo$ is  completely order isomorphic  with $\S$ via $a \mapsto \jlim_n \psi_n(a)$ (the inverse is $\jlim_n a_n \mapsto \lim_n \varphi_n(a_n)$).
    With this identification of $\S$ and $\S_\oo$ we get $\j\oo n=\varphi_n$. 

  When $\S$ is separable, 
   the index set can be chosen to be $\mathbb{N}$ by sequentially choosing better and better cpc approximations for a nested sequence of finite subsets of $\S$ whose union is dense in $\S$. One can then use a compactness argument (such as in \cite[Remark 3.2(ii)]{CW1}), to further refine this resulting sequence $\S\xrightarrow{\psi_n}M_{\nu_n}\xrightarrow{\varphi_n}\S$ of cpc approximations to one satisfying 
   \[\|\varphi_m-\varphi_n\circ \psi_n\circ\varphi_m\|<\frac{1}{2^n}, \text{ for all } n>m\geq 0.\]   (This is called a \emph{summable} system of cpc approximations in \cite{CW1}.) 
   Summability of the cpc approximations 
   implies summability of the induced soft inductive system $j_{nm}=\psi_n\circ\varphi_m:M_{\nu_m}\to M_{\nu_n}$. And so, as in Remark~\ref{rem:summable to coherent}, we can form a strict system $\oj{nm}:=\psi_n\circ\varphi_{n-1}\circ\hdots \circ \psi_{m+1}\circ\varphi_m:M_{\nu_m}\to M_{\nu_n}$ with the same limit.
\end{proof}
By \cite[Remark 5.15]{kavruk2011tensor}, a unital $\Cstar$-algebra is nuclear as an operator system if and only if it is nuclear as a $\Cstar$-algebra. Hence, we have the following. 
\begin{corollary}
    Every unital nuclear $\Cstar$-algebra $\A$ is a soft inductive limit of matrix algebras in the category of operator systems.
    This inductive system can be chosen to be split. 
    If $\A$ is separable, the directed set can be chosen to be $N=\NN$.\footnote{In this case, the complete order isomorphism in the proof above is essentially the map $\Psi$ in \cite{CW1} and \cite{C23}.}
\end{corollary}

When the inductive limits are split, we immediately get a converse.

\begin{proposition}\label{prop:split nuc}
    Split inductive limits of nuclear operator systems are nuclear.
    In particular, split inductive limits of nuclear $\Cstar$-algebras are nuclear. 
\end{proposition}

\begin{proof}
    Let $(\S,j,s)$ be a split inductive system.
    Given $a\up 1,\ldots,a\up k\in\S_\oo$ and $\eps>0$ set $a\up i _n = s_n(a\up i)$.
    Pick $m$ so that $\norm{a\up i - \j\oo m a\up i_m}<\eps /3$ for all $i$.
    Now use the nuclearity of $\S_m$ to pick $\alpha:\S_m\to M_t$ and $\beta:M_t\to\S_m$, $t\in\NN$, so that  $\norm{a_m\up i - \beta\circ\alpha(a_m\up i)}<\eps /3$ for all $i$.
    Now set $\gamma = \alpha\circ s_m$ and $\theta = \j\oo m \circ\beta$ and use the triangle inequality to check that $\norm{a\up i -\theta\circ\gamma(a\up i)}<\eps$ for all $i$.
\end{proof}

Also in the case of separable operator systems and strict inductive system, we have a converse to Theorem~\ref{thm:nuclearity_iff}. 

\begin{theorem}\label{thm: strict iff}
    Let $\S$ be a separable operator system. Then $\S$ is nuclear if and only if it is the strict inductive limit of matrix algebras, i.e., there is a strict inductive sequence $(\S,j)$ with each $\S_n$ being a matrix algebra so that $\S$ is its limit space. 
\end{theorem}

 An argument for the converse is essentially contained in the proof of \cite[Theorem 4.6]{CP23}, which pushes some classic results for $C^*$-algebras to the operator system setting. We generalize and flesh out the arguments here in a series of lemmas:

\begin{lemma}
       Let $\A$ be a $C^*$-algebra with two-sided closed ideal $\J$, and let $\S\subset \A$ be an operator system which contains $\J$. Then $(\S/\J)^{**}$ is completely order isomorphic to $p^\perp \S^{**}$ where $p$ is the central support projection of $\J$. (In fact, we have $\S^{**}\cong \J^{**}\oplus (\S/\J)^{**}$.)
\end{lemma}

\begin{proof}

For each $\varphi\in S(\A)$, write $(\sigma_\varphi, H_\varphi, \xi_\varphi)$ for the associated cyclic GNS representation. Write $H_\A:=\bigoplus_{\varphi\in S(\A)} H_\varphi$ and let 
$\sigma_u:\A\to B(H_\A)$ be the universal representation for $\A$. Since $\J$ is an ideal in $\A$, its essential support $\overline{\sigma_u(\J)H_\A}$  is invariant under $\sigma_u(\A)$, and so its orthogonal projection $p$ (which is the SOT limit of an increasing approximate identity in $\J$) lies in $\sigma_u(\A)'\cap \sigma_u(\A)''$. Hence we have 
\[\sigma_u(\A)''=p\sigma_u(\A)''\oplus p^\perp \sigma_u(\A)''.\]
Note that the orthogonal complement of $\overline{\sigma_u(\J)H_\A}$ is $\bigoplus_{\J\subset \text{ker}(\sigma_\varphi), \varphi\in S(\A)} H_\varphi$ (i.e., the largest subspace on which $\sigma_u(x)=0$ for all $x\in \J$).

We can identify $\A/\J$ with $p^\perp \sigma_u(\A)$ via the map $\pi:\A/\J\to B(p^\perp H_\A)$ given by $\pi(a+\J)=\oplus_{\J\subset \text{ker}(\sigma_\varphi), \varphi\in S(\A)}$. Since every (GNS) representation of $\A/\J$ extends to a (GNS) representation of $\A$ which vanishes on $\J$ (by composing with the quotient map), we can identify $\pi(\A/\J)$ with the image of $\A/\J$ under its universal representaion. In particular, that means we can identify $(\A/\J)^{**}\cong \pi(\A/\J)''=p^\perp \sigma_u(\A)''$. We can also write $p\sigma_u(\A)''\cong \J^{**}$ (though that is not important here), and, by identifying $\sigma_u(\A)''\cong \A^{**}$, we have 
\[\A^{**}\cong p\A^{**}\oplus p^\perp \A^{**}\cong \J^{**}\oplus (\A/\J)^{**}.\]

Since $\S\subset \A$, we have $p^\perp \S^{**}\subset p^\perp \A^{**}$. We just need to argue that $p^\perp \S^{**}$ is mapped to $(\S/\J)^{**}$ under the identification $p^\perp \A^{**}\cong (\A/\J)^{**}$. Since $p^\perp \S=\pi(\S/\J)$, we know that $p^\perp \S$ gets sent to $\S/\J$ in this identification. Since $\S$ and $\S/\J$ are ultraweakly (wk$^*$) dense in $\S^{**}$ and $(\S/\J)^{**}$, respectively, 
and since multiplication by a central projection and the above identification $p^\perp \A^{**}\cong (\A/\J)^{**}$ are ultraweakly (wk$^*$) continuous, it follows that 
$p^\perp \S^{**}$ gets sent to $(\S/\J)^{**}$ under the identification $p^\perp \A^{**}\cong (\A/\J)^{**}$. Since this identification is a $^*$-isomorphism, its restriction and co-restriction gives a complete order isomorphism between $(\S/\J)^{**}$ and  $p^\perp \S^{**}$. 

\end{proof}

\begin{lemma}\label{lem: nuc quotient}
    Let $\A$ be a $C^*$-algebra with two-sided closed ideal $\J$, and let $\S\subset \A$ be a separable nuclear operator system which contains $\J$. Then $\S/\J$ is nuclear. 
\end{lemma}

\begin{proof}
    It follows from \cite{CE77,EOR01,Kir95} that an operator system is nuclear if and only if its double dual is completely order isomorphic to an injective von Neumann algebra. (See \cite[Theorem 3.5]{han2011approximation} for a proof.) Then $\S^{**}$ is injective. Letting $p$ denote the central support projection of $\J$, we have $\S^{**}\cong p\S^{**}\oplus p^\perp \S^{**}$, and so compression by $p^\perp$ is a conditional expectation onto $p^\perp \S^{**}$. Then $p^\perp \S^{**}$ is also injective, and hence so is $(\S/\J)^{**}$. Since it is injective, we know from \cite{CE77} that it is completely order isomorphic to a $C^*$-algebra, which is moreover a double dual and thus a von Neumann algebra by Sakai's theorem \cite[Thm.~1.16.7]{sakai_c-algebras_1998}. 
    Hence $(\S/\J)^{**}$ is completely order isomorphic to a von Neumann algebra, which (by \cite[Lemma 2.8 (ii)]{Kir95}) implies that $\S/\J$ is a nuclear operator system. 
\end{proof}

\begin{remark}
    This is not true for general quotients of operator systems. In \cite[Section 2]{FP12}, Farenick and Paulsen consider the operator system $M_n/\J_n$ for $n>2$ where \[\J_n=\{d\in M_n\mid d\text{ diagonal and } \text{tr}_n(D)=0\}.\] They show (\cite[Theorem 2.6]{FP12}) this is completely order isomorphic to the operator system 
    \[W_n=\text{span}\{w_iw_j^*\mid 1\leq i,j\leq n\}\subset C^*(\mathbb{F}_{n-1})\] where $w_2,...,w_n$ are the unitaries in $C^*(\mathbb{F}_{n-1})$ corresponding to the generators of $\mathbb{F}_{n-1}$ and $w_1=1$. Moreover, in the same theorem, they show that the enveloping $C^*$-algebra $C^*_e(W_n)$ is $C^*(\mathbb{F}_{n-1})$. Since $W_n\subset C^*(\mathbb{F}_{n-1})$ has enough unitaries (meaning $\mathcal{U}(C^*(\mathbb{F}_{n-1}))\cap W_n$ generates $C^*(\mathbb{F}_{n-1})$ as a $C^*$-algebra) and since $C^*(\mathbb{F}_{n-1})$ is non-exact, it follows from \cite[Corollary 9.6]{kavruk2013quotients}, that $W_n$ is not (min,el)-nuclear. By the hierarchy of operator system tensor products, this implies that $W_n$ is also not (max,min)-nuclear, which we here call nuclear.  
\end{remark}

\begin{lemma}\label{lem: preimage is nuc}
Let $(\S,j)$ be a strict inductive sequence such that $\S_n=M_{k_n}$ for each $n\in \mathbb{N}$. Then $\C(\S\d,j)$ is a nuclear operator subsystem of $\nets(\S\d)=\prod_n M_{k_n}$.
\end{lemma}

\begin{proof}
To show $\C(\S\d,j)\subset \nets(\S\d)=\prod_n M_{k_n}$ is nuclear, we present a sequence of completely positive approximations through matrix algebras. For each $m\in \mathbb{N}$, let $\psi_m:\C(\S\d,j)\to \bigoplus_{j=1}^m M_{k_j}$ be the restriction of the canonical compression $\prod_n M_{k_n}\to \bigoplus_{j=1}^m M_{k_j}$ to $\C(\S\d,j)$, and let $\varphi_m:\bigoplus_{j=1}^m M_{k_j}\to \prod_n M_{k_n}$ be defined for $(x_1,...,x_m)\in \bigoplus_{j=1}^m M_{k_j}$ as 
    \[(x_1,...,x_m)\mapsto (x_1,...,x_m,\j_{m+1,m}(x_m),\j_{m+2,m}(x_m),...).\]
   Note that we can identify $\C(\S\d,j)$ with the closure of the set 
\[X_0:=\{(x_n)_n\in \prod_n M_{k_n} : \exists\ m\geq 1, x_{n+1}=\j_{n+1,n}(x_n), \forall\ n\geq m\}.\]
Hence the image of $\varphi_m$ clearly lies in $\C(\S\d,j)$, and to check nuclearity, we only need to check it for $(x_n)_n\in X_0$. For such a sequence, we choose $m>0$ such that for all $n\geq m$, $\j_{n+1,n}(x_n)=x_{n+1}$. Then for all $n\geq m$, we have $\varphi_n(\psi_n(x))=x$. It follows that $\C(\S\d,j)$ is nuclear.
\end{proof}

\begin{proof}[Proof of Theorem~\ref{thm: strict iff}]
    The forward direction is Theorem~\ref{thm:nuclearity_iff}. For the converse, we recall from Proposition~\ref{prop: limit as quotient} that the limit $\S$ of the strict system $(\S,j)$ is completely order isomorphic to $\C(\S\d,j)/\C_0(\S\d)$ identified as a subspace of $\nets(\S\d)/\C_0(\S\d)$ where here $\nets(\S\d)=\prod_n M_{k_n}$ and $\C_0(\S\d)=\bigoplus_n M_{k_n}$. By Lemma~\ref{lem: preimage is nuc}, $\C(\S\d,j)$ is nuclear, and hence by Lemma~\ref{lem: nuc quotient}, so is $\C(\S\d,j)/\C_0(\S\d)$.
\end{proof}

\begin{remark}
    Note that Theorem~\ref{thm: strict iff} also works with matrix algebras replaced by finite-dimensional von Neumann algebras, i.e., algebras of the form $\bigoplus_{j=1}^m M_{n_j}$.
\end{remark}


\subsection{Noncommutative Choquet theory}\label{sect. nc}
In this section, we provide an interpretation of Theorem~\ref{thm: strict iff} in the context of noncommutative Choquet theory.  We start by surveying the relevant classical results and framework and followed by a brief account of the noncommutative Choquet theory of Davidson and Kennedy (\cite{Davidson2016, DK24}).

For a compact, convex, metrizable set $K$, we consider the $^*$-vector space of its continuous $\mathbb{C}$-valued affine functions $A(K)$, which we view as an operator subsystem of $C(K)$. Through Kadison duality over $\mathbb{C}$ \cite{VT09}, we have an affine homeomorphism between $S(A(K))$ and $K$ via point evaluations. 
Numerous authors have characterized when $K$ is a Choquet simplex via properties of $K$, and as observed in \cite[Section 4]{Davidson2016}, these have nice formulations in our setting as well. In particular, Namioka and Phelps \cite{NP69} proved that $K$ is a simplex precisely when $A(K)$ is nuclear in the function system category (meaning there is a unique Archimedean order unit space structure that one can put on its tensor product with any other function system), and Effros \cite{Eff67} proved that $K$ is a simplex if and only if $A(K)^{**}$ is injective in that category. It follows from \cite{J68} and \cite{DVS68} (independently) that the projective limit of finite-dimensional simplices with affine connecting maps is a simplex, and in \cite{LL71}, Lazar and Lindenstrauss proved that the converse also holds (Corollary to \cite[Theorem 5.2]{LL71}). Together, we have the classical case:

\begin{theorem}[Jellett, Davies--Vincent--Smith, Lazar--Lindenstrauss]
    A compact convex metrizable set is a simplex if and only if it is the projective limit of finite-dimensional simplices.
\end{theorem}
Dualizing this (and identifying $A(\Delta^n)$ with $\mathbb{C}^{n+1}$) we get:  
\begin{theorem}[Jellett, Davies--Vincent--Smith, Lazar--Lindenstrauss, Namioka--Phelps]
    A separable function system is nuclear if and only if it is the inductive limit of finite-dimensional commutative von Neumann algebras.
\end{theorem}

Now we move to the noncommutative convexity setting in the sense of Davidson and Kennedy \cite{Davidson2016, DK24}, which builds on the matrix convexity theory from \cite{Wit81,Wit84,EW97,WW99}.
 In lieu of a full account of the theory, we offer an intuition based on a categorical duality stemming from Davidson and Kennedy's noncommutative generalization of Kadison's representation theorem \cite[Theorem 3.2.3]{DK24}. We recommend \cite[Section 3]{DK24} and \cite{Davidson2016} to any reader desiring further details. 

To define an nc convex set, we start with an operator space $E$ and a cardinal $\kappa$ at least the size of some dense subset of $E$. A compact nc convex set is a graded subset $K=\bigsqcup_{n\leq \kappa} K_n\subseteq \bigsqcup_{n\leq \kappa} M_n(E)$ which is closed under arbitrary direct sums and compressions and where each $K_n\subseteq M_n(E)$ is compact. 
The key example of a compact nc convex set is the nc state space of an operator system $\mathcal{S}$ \cite[Example 2.2.6]{DK24}: 
\[S_{nc}(\mathcal{S}):=\bigsqcup_{n\leq \kappa} \text{UCP}(\mathcal{S},M_n)\subseteq \bigsqcup_{n\leq \kappa}M_n(\mathcal{S}^*),\] 
where $\text{UCP}(\mathcal{S},M_n)$ consists of all unital completely positive (u.c.p.) maps $\mathcal{S}\to M_n$ with $n$ ranging over all cardinals $n$ less than or equal to some $\kappa$, which is determined by the operator space $\mathcal{S}^*$ as above. For example, $S_{nc}(M_n)$ consists of all ucp maps $M_n\to B(\mathcal{H}_m)$ with $\mathcal{H}_m$ a fixed Hilbert space of dimension $m\leq \aleph_0$. A continuous nc affine map between nc convex sets is a graded map which is continuous at each level and preserves direct sums and compressions  (see \cite[Definition 2.5.1]{DK24}). For a compact nc convex set $K=\bigsqcup_{n\leq \kappa}K_n$, the space $A(K)$ of continuous nc affine functions from $K$ to $\bigsqcup_{n\leq \kappa} M_n$ is a closed operator system. 

Now we can give Davidson and Kennedy's noncommutative analogue of Kadison's representation theorem \cite[Theorem 3.2.3]{DK24}: if $\mathcal{S}$ is a closed operator system, then $A(S_{nc}(\mathcal{S}))$ is completely order isomorphic to $\mathcal{S}$. This leads to a duality between the category $\text{OpSys}$ of closed operator systems with u.c.p.\ maps and the category $\text{NCConv}$ of compact nc convex sets with continuous nc affine maps (\cite[Section 3]{DK24}). In particular by \cite[Theorem 3.2.5]{DK24}, there is a contravariant functor $A:\text{NCConv}\to \text{OpSys}$ with inverse $A^{-1}$ implementing the dual equivalence. This functor maps a compact nc convex set $K$ to $A(K)$ and a continuous nc affine map $\theta:K\to L$ to the u.c.p.\ map $A(\theta):A(L)\to A(K)$ given by 
\[A(\theta)(b)(x)=b(\theta(x)),\ \text{ for }\ b\in A(L), x\in K.\]
Its inverse maps an operator system $\mathcal{S}$ to $S_{nc}(\mathcal{S})$ and a u.c.p.\ map $\varphi: \mathcal{S}\to \mathcal{T}$ to the continuous nc affine map $A^{-1}(\varphi):S_{nc}(\mathcal{T})\to S_{nc}(\mathcal{S})$ given by
\[b(A^{-1}(\varphi)(x))=\varphi(b)(x),\ \text{ for }\ b\in \mathcal{S}, x\in S_{nc}(\mathcal{T}).\]
Thus from a strong inductive sequence of operator systems $(\S,j)$, we get a projective sequence of their noncommutative state spaces and vice versa.


In \cite{KS22}, Kennedy and Shamovich introduced nc (Choquet) simplices, which are characterized in \cite[Theorem 6.2]{KS22} to be the compact nc convex sets $K$ for which $A(K)^{**}$ is completely order isomorphic to a von Neumann algebra (i.e., a $\Cstar$-system in the language of \cite{KW98}). An extensive class of examples comes from the nuclear operator systems. Indeed, combining deep work of Kirchberg \cite{Kir95}, Choi--Effros \cite{CE77}, and Effros--Ozawa--Ruan \cite{EOR01} with Sakai's theorem \cite[Thm.~1.16.7]{sakai_c-algebras_1998} yields that an operator system $\mathcal{S}$ is nuclear if and only if $\mathcal{S}^{**}$ is completely order isomorphic to an \emph{injective} von Neumann algebra (see \cite[Theorem 3.5]{han2011approximation}). In finite dimensions, \cite{CE77} tells us $\mathcal{S}$ is nuclear if and only if $\mathcal{S}$ is itself completely order isomorphic to a von Neumann algebra (and hence to some finite direct sum of matrix algebras). Since the space $A(\Delta^n)$ of continuous $\mathbb{C}$-valued affine functions on a classical $n$-simplex can be identified with $\mathbb{C}^{n+1}$ (with each function determined by where it sends $\partial_e(\Delta^n)$), this suggest nc state spaces of finite-dimensional von Neumann algebras as the nc analogue of finite-dimensional simplices.  

There is now more nuance from the classical function system setting in the sense that being the affine function space of an nc simplex is more general than being nuclear. Hence for clarity we call the nc state space for a nuclear operator system a \emph{nuclear nc simplex}. 
Hence we arrive at dual statement of Theorem~\ref{thm: strict iff}:

\begin{corollary}\label{cor: strict iff*}
    Let $K$ be a compact nc convex set with separable $A(K)$. Then $K$ is a nuclear nc simplex if and only if $K$ is the projective limit of nc state spaces of finite-dimensional von Neumann algebras. 
\end{corollary}

\printbibliography

\end{document}